\documentclass{amsart}
\theoremstyle{plain}

\usepackage[multiple]{footmisc}
\usepackage{amsthm}
\usepackage{amssymb}
\usepackage{amsmath}
\usepackage{xfrac}
\usepackage{array}
\usepackage{mathtools}
\usepackage{stmaryrd}

\newtheorem{prop}{Proposition}	
\newtheorem{aprop}{Proposition}	

\newtheorem{theorem}[prop]{Theorem}
\newtheorem{atheorem}[aprop]{Theorem}
\newtheorem{cor}[prop]{Corollary}
\newtheorem{lemma}[prop]{Lemma}

\newtheorem*{definition*}{Definition}
\newtheorem*{theorem*}{Theorem}
\newtheorem*{remark*}{Remark}

\newcommand{\RR}{\ensuremath{\mathbb{R}}}
\newcommand{\CC}{\ensuremath{\mathbb{C}}}

\newcommand{\UU}{\ensuremath{\mathbb{U}}}

\newcommand{\GG}{\ensuremath{\mathbb{G}}}

\newcommand{\TT}{\ensuremath{\mathbb{T}}}
\newcommand{\BB}{\ensuremath{\mathbb{B}}}

\newcommand{\NN}{\ensuremath{\mathbb{N}}}
\newcommand{\ZZ}{\ensuremath{\mathbb{Z}}}

\newcommand{\Qp}{\ensuremath{\mathbb{Q}_{p}}}

\newcommand{\HH}{\ensuremath{\mathbb{H}}}

\newcommand{\IRS}[1]{\ensuremath{\mathrm{IRS}\left( #1 \right) }}

\newcommand{\SL}[2]{\mathrm{SL}_{#1}\left(#2\right)}

\newcommand{\dd}{\ensuremath{\, \textrm{d}}}

\DeclarePairedDelimiter\abs{\lvert}{\rvert}
\DeclarePairedDelimiter\norm{\lVert}{\rVert}

\newcommand{\nrm}{\ensuremath{\vartriangleleft }}

\newcommand{\Sub}[1]{\ensuremath{\textrm{Sub}\left(#1\right)}}

\title{On  Benjamini--Schramm limits of congruence subgroups }

\author{Arie Levit}

\begin{document}

\maketitle

\begin{abstract}
A sequence of orbifolds corresponding to pairwise non-conjugate congruence lattices in a higher rank semisimple group over zero characteristic local fields is Benjamini--Schramm convergent to the universal cover.
\end{abstract}

%Let $X$ be a compact space and $Let $\mathcal{M}_1(X)$ denote the space of Borel probability measures on the compact space $X$.

\section{Introduction}

A \emph{semisimple analytic group} $G$ is defined as follows. 
Let $I$ be a finite index set. Assume that  $k_i$ is a zero characteristic local field 
and $\GG_i$ is a connected simply-connected $k_i$-isotropic almost $k_i$-simple linear $k_i$-group for every  $ i \in I$. Denote $G_i = \GG_i(k_i)$ so that in particular  $G_i$ is an almost simple non-compact linear group admitting a $k_i$-analytic structure.  Let $G = \prod_{i\in I}G_i$.

\begin{definition*}
	\label{def:benjamini-schramm convergence}
	A sequence of lattices $\left(\Gamma_n\right)_{n\in \NN}$ in $G$ is called weakly central\footnote{Of course, this definition makes sense for any locally compact group.}
	% In the case that $G$ is totally disconnected a weakly central sequence is  called a Farber sequence.} 
	if for every compact subset $Q \subset G$  we have that
	$$ \eta_n \left(\{ g\Gamma_n \in G/\Gamma_n \; : \; g \Gamma_n g^{-1} \cap Q \subset Z(G)  \}\right) \xrightarrow{n\to\infty} 1$$
	where $\eta_n$ is the $G$-invariant probability measure on $G/\Gamma_n$ for each $n \in \mathbb{N}$.
\end{definition*}

%A geometric interpretation of a sequence of lattices being weakly central is discussed below. 
This note is dedicated to establishing the following result.

\begin{theorem}
\label{thm:main theorem}
Assume that $\abs{I} \ge 2$. Then every  sequence of pairwise non-conjugate congruence lattices in $G$ is weakly central.	
\end{theorem}

Recall that by the celebrated theorem of Margulis every irreducible lattice in $G$ is arithmetic whenever $\abs{I} \ge 2$. A
 \emph{congruence lattice} is a particular kind of an irreducible arithmetic lattice   containing a principal congruence subgroup.  See \S\ref{sec:congruence lattices and the selberg property} for a precise definition of this notion. In particular,  whenever lattices in $G$ are known to satisfy the congruence subgroup property a stronger formulation of Theorem \ref{thm:main theorem} is made possible.

%\begin{cor}
%\label{cor:main theorem with CSP}
%Assume that  $\abs{I} \ge 2$. Then  every  sequence of pairwise non-conjugate irreducible lattices  having the congruence subgroup property is weakly central.
%\end{cor}

%Examples of semisimple analytic groups in which all lattices are known to have the congruence subgroup property include products of $\mathrm{SL}_n, n \ge 3$ or of $\mathrm{Sp}_n, n \ge 2$ with some $k_i \neq \CC$ \cite{bass1967solution}, or any group with some $k_i$  non-Archimedean \cite[p. 268]{Ma}.

We  remark that if $\abs{I} = 1$ and $\mathrm{rank}(G) \ge 2$ then every sequence of pairwise non-conjugate  lattices is weakly central by the main results of \cite{7S,GL}. 
The recent works of Raimbault \cite{raimbault2013convergence} and Fraczyk \cite{fraczyk2016strong} establish closely related results for congruence lattices in the rank one  groups $\SL{2}{\RR}$ and $\SL{2}{\CC}$. We also mention 
\cite[\S 5]{7S} dealing with  congruence subgroups in a fixed uniform arithmetic lattice.

\subsection*{Benjamini--Schramm convergence}

The semisimple analytic group $G$ is acting by isometries on a contractible non-positively curved metric space $X$,  as follows. Let $X_i$ be the symmetric space or Bruhat-Tits building associated to $G_i$ for every $i \in I$, depending on whether $k_i$ is Archimedean or not.  Take $X = \prod_{i\in I}X_i$ equipped with the product metric.

Let $\left(\Gamma_n\right)$ be a sequence of lattices in $G$. The following geometric notion  is equivalent to saying that $(\Gamma_n)$ is weakly central. See \cite[\S 3]{GL} for more details.

\begin{definition*}
	 The orbifolds $\Gamma_n \backslash X$ Benjamini--Schramm converge to $X$ if for every radius $0 < R < \infty$ the probability that an $R$-ball in $\Gamma_n \backslash X$ with base point taken uniformly at random is contractible tends to one as $ n \to \infty$.
\end{definition*}

As an example, we provide a geometric application of Theorem \ref{thm:main theorem} to arithmetic  orbifolds, relying on the congruence subgroup property  \cite{SerCong}. 

\begin{cor}
	\label{cor:case of hyperbolic factors}
	Let $F$ be a number field with ring of integers $\mathcal{O}_F$. Assume that $F$ has $r$ real embeddings and $2s$ complex embeddings with $r + s \ge 2$. Consider the irreducible arithmetic lattice
	$$  \SL{2}{\mathcal{O}_F} \hookrightarrow \prod_{i=1}^r \SL{2}{\RR} \times \prod_{i=1}^s \SL{2}{\CC}. $$
	Let $X = (\HH^2)^r \times (\HH^3)^s$ be a  product of two and three dimensional hyperbolic spaces.  Then the orbifolds corresponding to any sequence of  distinct finite-index subgroups of   $\SL{2}{\mathcal{O}_F}$ are Benjamini-Schramm convergent to $X$. 
\end{cor}

The conclusion of Corollary \ref{cor:case of hyperbolic factors} continues to hold if one moreover varies $F$ among all number fields with the same signature.

%See Corollary \ref{cor:case of hyperbolic factors} for yet another example. 

\subsection*{On properties $(T)$ and $(\tau)$} Benjamini-Schramm convergence for lattices was first investigated in \cite{7S}. It is shown in \cite{7S} that any sequence of pairwise non-conjugate irreducible lattices in a  semisimple Lie group with high rank and property $(T)$ is weakly central.  General local fields are dealt with in \cite{GL}.  These proofs rely on property $(T)$, most crucially in order to invoke the Stuck--Zimmer theorem \cite{LIFT,SZ}.

Our approach is to make use of \emph{property $(\tau)$} instead, avoiding the Stuck--Zimmer theorem which is presently unknown in the absence of property $(T)$. More precisely, we rely on property $(\tau)$ with respect to congruence lattices. This is sometimes called the \emph{Selberg property} as it generalizes his famous theorem on congruence subgroups of the modular group. It is crucial  that Selberg's property is   uniform  with respect to all the congruence lattices inside $G$ --- see Theorem \ref{thm:Clozel}. 
%\emph{I would like to thank Nicolas Bergeron for bringing this important fact to my attention.}

In addition, we rely on topological properties  of the Chabauty space of semisimple analytic groups  recently established  by Gelander and the author \cite{GL}, thereby replacing  yet another argument of \cite{7S} which previously required property $(T)$.

\subsection*{Spectral gap and essentially free actions}

Towards proving Theorem \ref{thm:main theorem} we study a  Borel $G$-space obtained by taking a certain limit with respect to a sequence of congruence lattices. Selberg's property implies that this limiting $G$-space has spectral gap --- see \S \ref{sec:uniform spectral gap} for a discussion of this notion. The following theorem allows us to deduce that such an action is essentially free, provided that $\abs{I} \ge 2$.

\begin{theorem}
\label{thm:spectral gap action is essentially free}
Let $G$ be a product of at least two locally compact simple groups and $X$ a Borel $G$-space admitting an invariant probability measure $\mu$. Assume that $X$ is properly ergodic, irreducible and has spectral gap. Then $(X,\mu)$ is essentially free.
\end{theorem}

We do not claim any originality for Theorem \ref{thm:spectral gap action is essentially free} --- it is a formal corollary of the well-known work of Bader--Shalom \cite{BS} and the fact that an action with spectral gap is not weakly amenable \cite{C_L}. We state it here merely as an observation and in the hope that it may prove useful in other situations as well.

Notice  that for Theorem \ref{thm:spectral gap action is essentially free} to be applicable in our situation the limiting $G$-space has to be irreducible. Indeed, this follows from a key feature of Theorem \ref{thm:Clozel} establishing spectral gap with respect  to each simple  factor of $G$ individually. 

\subsection*{An application --- convergence of Plancherel measures}

Let $\nu^G$ denote the Plancherel measure on the unitary dual $\widehat{G}$ of $G$. For every uniform lattice $\Gamma$ in $G$, the quasi-regular representation $\rho_\Gamma$ of $G$ in $L^2(G / \Gamma)$  decomposes as a direct sum of irreducible representations. Every irreducible representation $\pi \in \widehat{G}$  appears in $\rho_\Gamma$ with finite multiplicity $m(\pi,\Gamma)$. The corresponding \emph{relative Plancherel measure}  is 
$$\nu_\Gamma = \frac{1}{\mathrm{vol}(G / \Gamma)} \sum_{\pi \in \widehat{G}} m(\pi, \Gamma) \delta_\pi $$

Combining Theorem \ref{thm:main theorem} with \cite[1.2]{7S} and \cite[1.3]{GL} we obtain a generalization of one of the main results of \cite{7S,GL} regarding convergence of relative Plancherel measures. 

\begin{cor}
	\label{cor:Planceherel measure}
	Assume that $\abs{I} \ge 2$. Let $\Gamma_n$ be any  sequence of pairwise non-conjugate uniformly discrete torsion-free congruence lattices in $G$. 
	Then 
	$ \nu_{\Gamma_n}(E) \xrightarrow{n \to \infty} \nu^G(E) $
	for every relatively quasi-compact  $\nu^G$-regular  subset $E \subset \widehat{G}$.
\end{cor}

Further applications of  Corollary \ref{cor:Planceherel measure}  to limit multiplicities formulas and normalized Betti numbers  \cite[1.3,1.4]{7S} carry over to our setting as well.

\subsection*{Acknowledgements}

I would like to thank Tsachik Gelander for his guidance and many useful discussions about the ideas of \cite{7S}. 
I would like to thank Nicolas Bergeron for bringing to my attention the important fact that Clozel's theorem \cite{clozel2003demonstration} holds uniformly with respect to all congruence lattices. 
I would like to thank Uri Bader and Alex Lubotzky for several useful discussions and remarks.

%
%t.he Benjamini-Schramm limit of $G/\Gamma_i$ this limit has to be irreducible. The irreducibility follows from a certain stronger form of property $(\tau)$ related to \emph{strong spectral gap}, i.e. spectral gap with respect to the restricted action of ; see Sec \ref{sec:strong spectral gap}.

%\begin{remark*}
%The simple connectedness assumption can be eliminated from Theorem \ref{thm:main theorem}. However, this requires  studying a certain subgroup $G^+ \nrm G$ and introduces various technical subtleties \cite[I.1.5, I.2.3]{Ma}. We have chosen to sidestep this issue.
%\end{remark*}
%\begin{prop}
%Let $G$  be a product of two locally compact simple groups. Let $\Gamma \le G$ be a lattice having strong property $(\tau)$ with respect to a family $\mathcal{F}$ of finite index subgroups. Assume that lattices in $G$ have Chabauty open conjugacy classes. Then the invariant random subgroups $\mu_{N_i}$ weakly converge to $\mu_{\{e\}}$.
%\end{prop}

%\subsection*{Acknowledgments}

\section{Uniform spectral gap}
\label{sec:uniform spectral gap}

Let $G$ be a compactly generated locally compact group and $(X,\mu)$  a Borel $G$-space with an invariant probability measure. 
 Recall that the $G$-action on $X$ has \emph{spectral gap} if the Koopman representation of $G$ in the space $L^2_0(X,\mu)$ of  functions with zero integral does not almost admit  invariant vectors. 
%\footnote{A lattice $\Gamma \le G$ such that the $G$-space $G/\Gamma$ has spectal gap is  called \emph{weakly cocompact}.}.

%\subsection{Uniform spectral gap}

\begin{definition*}
A sequence $(X_n,\mu_n)$ of Borel $G$-spaces with invariant probability measures has \emph{uniform spectral gap} if the natural representation of $G$ on $\oplus_{n} L^2_0(X_n,\mu_n) $  does not almost admit  invariant vectors. 
\end{definition*}

%In other words, the sequence $(X_n, \mu_n)$ has uniform spectral gap if there is an strongly adapted and absolutely continuous probability measure $\nu$ on $G$ and a constant $\varepsilon > 0$ so that $\norm{\pi_n(\nu)} < 1- \varepsilon$ for all $ n \in \NN$.

% compact subset $Q \subset G$ and a constant $\varepsilon > 0$ so that
%$$ \sup_{g\in Q} \norm{\pi_n(g)f_n - f_n} \ge \varepsilon \norm{f_n} $$
%for all $f_n \in L^2_0(X_n,\mu_n)$ and all $ n\in\NN$. 

For example, if $G$ has property $(T)$ then any family of ergodic $G$-invariant probability measures has uniform spectral gap. More generally, such uniformity is useful when passing to weak-$*$ limits of probability measures on a given compact $G$-space.

\begin{prop}
\label{prop:uniform spectral gap gives spectral gap in the limit}
Let $X$ be a compact $G$-space and $\mu_n$   a sequence of invariant Borel probability measures on $X$ with uniform spectral gap.  If $\mu$ is a weak-$*$ limit of the sequence $\mu_n$ then $\mu$ has spectral gap.
\end{prop}

Let $\pi$ and $\pi_n$ denote the Koopman representations on the  Hilbert spaces $L^2(X,\mu)$ and $L^2(X,\mu_n)$ for $n\in\NN$. Let $\norm{\cdot}$ and $\norm{\cdot}_n$ denote the  norms  coming from the Hilbert space structure on these spaces.

\begin{proof}[Proof of Proposition \ref{prop:uniform spectral gap gives spectral gap in the limit}]
%Let $\norm{\cdot}_{\mu_n}$ and $\norm{\cdot}_\mu$ denote the norms in the Hilbert spaces $L^2_0(X,\mu_n)$ and $L^2_0(X,\mu)$ respectively. 
Let $\nu$ be a probability measure on $G$ which is symmetric, absolutely continuous with respect to the Haar measure and such that $\mathrm{supp}(\nu * \nu)$ is a generating set for $G$. Uniform spectral gap for the measures $\mu_n$ means that   
$$\norm{\pi_n(\nu)_{|L^2_0(X,\mu_n)}}_n \le \beta $$ for some constant $0 < \beta < 1 $ and  all $n \in \NN$. This fact is established in \cite[G.4.2]{BdlHV}. 
%Likewise $\mu$ has spectral gap if and only if $\norm{\pi(\nu)} < 1$. 

%$\mu$ has spectral gap  respect 
%
%Since $G$ is compactly generated it admits a probability measure $\nu$ which is symmetric and absolutely continuous with respect to the Haar measure such that  is a generating set for $G$. Spectral gap forUniform spectral gap 

We claim that $\norm{\pi(\nu)_{|L^2_0(X,\mu)}} \le  \beta$ as well. In estimating the norm of a continuous operator we may restrict attention to a dense subspace.  Consider  any non-zero continuous function  $f \in C(X) \cap L^2_0(X,\mu)$. Note that $\pi(\nu)f \in C(X) \cap L^2_0(X,\mu)$ as well. Denote
$$a_n = \int f  \dd \mu_n, \quad f = f'_n + a_n 1_X $$
so that $f'_n \in C(X) \cap L^2_0(X,\mu_n)$ for all $n\in \NN$. 

Since $\mu$ is a weak-$*$ limit of the measures $\mu_n$ we have that $\lim_n a_n = \int f \dd \mu = 0$.
%$$ \lim_n \abs{a_n} = \abs{\int h  \dd \mu} \quad \text{and} \quad  \lim_n \norm{h}_{n} = \norm{h}   $$
%It follows that 
%$$\limsup_n \norm{f'_n}_{n}  \le \lim_n \norm{f}_{n} + \lim_n \abs{a_n} = \norm{f} + \abs{\int f  \dd \mu} = \norm{f}$$
To estimate the operator norm of $\pi(\nu)$ on the space $L^2_0(X,\mu)$ we calculate
%Consider $\pi(\nu)$ as an operator on the Banach space $C(X)$ so that
%$$ \pi(\nu)h_n = \pi(\nu)h'_n + a_n 1_X $$
%for all $ n \in \NN$. This gives
\begin{multline*}
 \norm{\pi(\nu) f} = \lim_n \norm{\pi(\nu) f}_{n} \le \limsup_n \norm{\pi(\nu)f'_n}_{n} + \lim_n \abs{a_n} \le \\
\le \beta \limsup_n \norm{f'_n}_{n} \le \beta\left(\lim_n \norm{f}_{n} + \lim_n \abs{a_n}\right) = \beta\norm{f}.
\end{multline*}
Therefore the $G$-space  $(X,\mu)$ has spectral gap as well  \cite[G.4.2]{BdlHV}.
\end{proof}

Assume that $G$ splits as a direct product $G = G_1 \times \cdots \times G_k$ of $k$-many factors. It is natural to consider the restriction of the $G$-action to each factor $G_i$ individually. For instance, the $G$-action is \emph{irreducible} if each   $G_i$ is acting ergodically.
 
 \begin{definition*}
 	$(X,\mu)$ has \emph{strong  spectral gap} if the restricted action of each factor 	has spectral gap.
 	 	 	A sequence $(X_n,\mu_n)$ of Borel probability $G$-spaces has  \emph{strong uniform  spectral gap} if these restricted actions have uniform spectral gap.
 \end{definition*}
 
 Since spectral gap  implies ergodicity, strong spectral gap implies irreducibility.

\section{Congruence lattices and the Selberg property}
\label{sec:congruence lattices and the selberg property}

%We now rigorously define congruence lattices and reformulate the Selberg property in terms of strong uniform spectral gap.

Let $G$ be semisimple analytic group. As in the first paragraph of the introduction $G$ is the direct product of the analytic groups $G_i = \GG_i(k_i)$ where every $\GG_i$ is a connected almost $k_i$-simple\footnote{A linear   $k$-group $\GG$ is \emph{almost $k$-simple} if $\GG$ is semisimple and   admits no proper normal $k$-subgroups of positive dimension.} linear group and $k_i$ a local field for every index  $i \in I$. 

\subsection*{Congruence lattices}

Let $F$ be an algebraic number field and $\HH$  an absolutely almost simple\footnote{A linear $k$-group $\GG$ is \emph{absolutely almost simple} if $\GG$ is almost  simple over the algebraic closure of $k$.}  linear $F$-group. Let $R  $ denote the inequivalent Archimedean valuations on the field $F$ such that $\HH(F_v)$ is non-compact for $v \in R$. Assume that there is  a finite set of valuations $S$ with  $R \subset S  $ and a bijection $\iota : I \to S$ so that $k_i \cong F_{\iota(a)}$ and $\GG_i$ is $k_i$-isomorphic to $\HH$ for all $i \in I$. In particular  we may identify $G$ with $\prod_{v \in S}\HH(F_v)$. 

Let $F(S)$ denote the ring of $S$-integers in the field $F$. The group $\HH(F(S))$ is  an irreducible lattice in $G$ by the Borel--Harish-Chandra theorem \cite[Theorem 4.8]{PR}. 
Given a non-zero ideal  $\mathfrak{a}$ in the ring $F(S)$  let $\HH(\mathfrak{a})$ denote the kernel of the natural map $\HH(F(S)) \to \HH(F(S)/\mathfrak{a})$. The subgroup $\HH(\mathfrak{a})$ is called a principal congruence subgroup.

\begin{definition*}
A \emph{congruence lattice} is  any lattice in $G$  containing some  $\HH(\mathfrak{a})$ as above.
\end{definition*}

%\footnote{The reader is referred to \cite[I.3, IX]{Ma} for more details on this construction.}. 

%The following celebrated result establishes this property for congruence subgroups of arithmetic groups.

%, and the following is essentially  \cite[Thm. 3.1]{clozel2003demonstration}.

%The following is a formulation of the well-known Selberg's property.

The following is essentially a reformulation of the well-known Selberg's property.

\begin{atheorem}[Selberg's property]
\label{thm:Clozel}
Let $G$ be a semisimple analytic group. Then the family of $G$-spaces $G/\Gamma$ with normalized probability measures and $\Gamma$ ranging over the congruence lattices in $G$ has strong uniform spectral gap.
\end{atheorem}

 We begin our discussion of Theorem \ref{thm:Clozel} with a few preliminary remarks.

\begin{itemize}
\item The two $G$-representations $L^2(X,\mu)$ and $L^2(X,\alpha \mu)$ are equivalent for every $\alpha > 0$, so that renormalizing a finite measure on a Borel $G$-space has no effect on spectral gap.  
\item Similarly, the two $G$-representations $L^2(G/\Gamma)$ and $L^2(G/\Gamma^g)$ are equivalent for every $g \in G$.%, and conjugating $\Gamma$ inside $G$ has no effect on spectral gap.
\item  For a pair of lattices $\Gamma, \Gamma'$ with $\Gamma \le \Gamma' $ the $G$-representation $L^2(G/\Gamma')$ is contained in $L^2(G/\Gamma)$.

\end{itemize}

In light of these remarks we may restrict our attention to the situation of the lattice $\HH(\mathfrak{a})$ inside $\prod_{v \in S}\HH(F_v)$ while making sure that the resulting strong spectral gap is independent of the field $F$, the group $\HH$ and the ideal $\mathfrak{a}$.

\subsection*{On Selberg's property}

%Let us provide an overview of 

The existence of spectral gap for congruence subgroups of $\SL{2}{\ZZ}$ regarded as lattices in $\SL{2}{\RR}$   is essentially Selberg's classical $\lambda_1 \ge \frac{3}{16}$ theorem \cite{selberg1965estimation}. The Archimedean case\footnote{Note that this case  already suffices for the purpose of our Corollary \ref{cor:case of hyperbolic factors}.} where $\HH$ is still $\mathrm{SL}_2$, $F$ is any number field and $S$ consists of infinite places is treated in \cite{vigneras1983quelques}. The remaining case of $\HH= \mathrm{SL}_2$ and $S$ an arbitrary set of places  follows from the work of Gelbart--Jacquet \cite{gelbart1978relation}.

The Burger--Sarnak method \cite{burger1991ramanujan} allows to go beyond  $\mathrm{SL} _2$. This method was extended by Clozel--Ullmo \cite{clozel2004equidistribution}, in particular covering the $p$-adic case. We refer the reader to the useful discussion on \cite[\S 4.2]{lubotzky2005property}.

\begin{atheorem}[Burger--Sarnak, Clozel--Ullmo]
\label{thm:Burger-Sarnak-Clozel-Ullmo}
Let $\HH_1$ be a semi-simple $F$-subgroup of $\HH$. Then for every valuation $v \in S$ the restriction of $L^2\left(G/\HH(\mathfrak{a})\right)$ to $\HH_1(F_v)$  is weakly contained in $$\bigoplus_{\mathfrak{a} ' \nrm F(S) } L^2\left(\prod_{v \in S}\HH_1(F_v)/\HH_1(\mathfrak{a}')\right)$$
where the direct sum is taken over all non-zero ideals in $F(S)$.
\end{atheorem}
	
Clozel made the final contribution towards Selberg's property by dealing with arbitrary 
absolutely simple groups. In fact Theorem \ref{thm:Clozel} is essentially equivalent to  \cite[Thm. 3.1]{clozel2003demonstration}. 
Clozel's proof for a general $F$-group $\HH$ depends on whether it is $F$-isotropic or not. If $\mathrm{rank}_F(\HH) \ge 1$ then $\HH$ is known  to admit $\mathrm{SL}_2$ as a $F$-subgroup \cite[I.1.6.3]{Ma}. In that case one may rely on \cite{gelbart1978relation} and Theorem \ref{thm:Burger-Sarnak-Clozel-Ullmo}.

The main effort of \cite{clozel2003demonstration} is in dealing with the anisotropic case, as follows. If $\HH$ is $F$-anisotropic then  it admits a $F$-subgroup $\HH_1$ with $\mathrm{rank}_{F_v}(\HH_1) = 1$ and such that
$$ \text{$\HH_1 \cong \SL{1}{D}$ \quad or \quad $\HH_1 \cong \mathrm{SU}(D,*)$} $$
where $D$ is a division algebra of degree $p^2$  over $F$ or over a quadratic extension of $F$ in the first and second cases, respectively, and $p$ is a prime \cite[1.1]{clozel2003demonstration}. 
The careful analysis of \cite[\S3.2]{clozel2003demonstration} reveals that the case of $\SL{1}{D}$ can be reduced to $\mathrm{SL}_2$. Similarly, the case of $\mathrm{SU}(D,*)$ can be reduced either to $\mathrm{SL}_2$, to $\mathrm{SU}(3,F_v)$ with $v$ finite or to $\mathrm{SU}(n,1)$ with $v$ infinite. The parameter $n$ is clearly bounded since $\HH_1(F_v)$ embeds in $G$. Clozel then establishes spectral gap directly for congruence lattices in these last two families of rank one groups, and concludes relying on Theorem \ref{thm:Burger-Sarnak-Clozel-Ullmo}.

\subsection*{Uniformity of spectral gap}

Note that while \cite[Theorem 3.1]{clozel2003demonstration} is stated with respect to a fixed algebraic number field $F$ and group $\HH$, the resulting spectral gap for the subgroup $\HH_1(F_v)$ turns out to be  independent of any such choices. 

To see this, observe that there are only finitely many possibilities for the group $\HH_1(F_v)$ and that the validity of Theorem \ref{thm:Clozel} for these implies the same for $G$. 
The fact that $\HH_1(F_v)$ regarded as a subgroup of $G$ depends on the chosen $F$-structure needs to be taken into account, relying on the above preliminary remarks and the following lemma. % Lemma \ref{lem:on embeddings of compact subsets}.

\begin{lemma}
\label{lem:on embeddings of compact subsets}
Let  $\HH$ and $\GG$ be a pair of almost $k$-simple linear groups over a local field $k$ of zero characteristic. Assume that $\mathrm{rank}_k(\HH) = 1$ and denote $H = \HH(k)$ and $G = \GG(k)$. Then for every compact subset $Q_1 \subset H$ there is a compact $Q \subset G$ such that every $k$-homomorphism $\varphi : \mathbb{H} \to \mathbb{G}$ satisfies $\varphi(Q_1)^g \subset Q$ for some $g \in G$. 
\end{lemma}
%The following argument is to be compared with Proposition 5 of \cite[VIII.\S11.3]{bourbaki2008lie}.
\begin{proof}%[Proof of Lemma \ref{lem:on embeddings of compact subsets}]
Consider   the Archimedean case first. Weil's restriction of scalars allows us to assume without loss of generality that the local field  $k$ is $\RR$.

 Let $\TT_1$ and $\TT$  be maximal $k$-split tori  in the  groups $\HH$ and $\GG$ respectively. Let $\BB$ be  a minimal parabolic $k$-subgroup  of $\GG$ containing the torus $\TT$ and with unipotent $k$-subgroup $\UU$. Since $\mathrm{rank}_k(\TT_1) = 1$  
the $k$-root system of $\HH$ relative to $\TT_1$ has type $\mathrm{A}_1$ and $\Phi_k(\HH,\TT_1) = \{\alpha, -\alpha\}$ for some $k$-character $\alpha \in X(\TT_1)$. Let $\UU_\alpha$ be the unipotent $k$-subgroup of $\HH$ corresponding to the root $\alpha$. This means that the Lie subalgebra $ \mathrm{Lie}(\UU_\alpha)$ is the weight space of $\alpha$ in the adjoint representation on the Lie algebra $ \mathrm{Lie}(\HH)$.
 
 Denote  $ T_1 = \mathbb{T}_1(k)$ and $T = \mathbb{T}(k)$. Let $K_1$ and $K$  be maximal compact subgroups of $H$ and of $G$ respectively,   admitting corresponding Cartan decompositions
$$ H = K_1 T_1 K_1 \quad \text{and} \quad G = KTK.$$
%In the non-Archimedean case let $C$ be a chamber $K_1,\ldots,K_l$ 

Let $\varphi : \HH\to \GG$ be any $k$-homomorphism.  We claim that there is an element $g \in G$ such that
$$  \varphi^{g}(K_1) \subset K, \quad \varphi^{g}(T_1) \subset T \quad \text{and} \quad \varphi^g(\UU_\alpha) \subset \UU.$$
To establish the claim we argue as follows. All maximal compact subgroups of $G$ are conjugate \cite[3.10]{PR}. Therefore there is an element $g_1 \in G$ such that $\varphi^{g_1}(K_1) \subset K$. Let $T'$ be any maximal $k$-split torus of $G$ containing   $\varphi^{g_1}(T_1)$. There is an element $g_2 \in K$ satisfying $(T')^{g_2} = T$, see  \cite[20.9, 24.7]{borel2012linear} or \cite[6.51]{knapp2013lie}.   The relative Weyl group of the torus $\TT$ is transitive on   the set of minimal parabolic $k$-subgroups containing $\TT$ \cite[21.3]{borel2012linear}, and every   element of this Weyl group can be realized in the maximal compact subgroup $K$  \cite[6.57]{knapp2013lie}. Hence there is an element $g_3 \in K \cap \mathrm{N}_G(T)$ such that $\varphi^{g}(\UU_\alpha) \subset \UU$ where $g = g_3 g_2 g_1$. This concludes the claim.

% The group $G$ admits a Cartan decomposition $G = K'T'K'$ for some  maximal torus $T'$ and   maximal compact subgroup $K'$ such that $T'_1 = T' \cap \varphi(H)$ is a maximal torus and $K'_1 = K' \cap \varphi(H)$ is a maximal compact subgroup of $\varphi(H)$ with Cartan decomposition $\varphi(H) = K'_1 T'_1 K'_1$. This statement is a consequence of Mostow's result \cite[Theorem 7.3]{mostow1955self}. Moreover, any two Cartan decompositions of $H$ and of $G$  are conjugate    \cite[Theorem 6.51]{knapp2013lie}. The claim follows, conjugating first by a corresponding element of $\varphi(H)$ and then by a corresponding  element of $G$.

%The torus Its unipotent radical $\UU$ is $k$-isomorphic as a variety to an affine space \cite[\S21.20]{borel2012linear}. This space is the direct product of the unipotent $k$-subgroups $\UU_\beta$ over all positive roots $\beta  \in \Phi_+(\GG,\TT)$. The $\TT$-action on $\UU$ preserves each factor, and its differential is given by the corresponding root.
%15.13, 20.9

% So there is some $g \in K g_1   $ such that $\varphi^g(K_1) \subset K, \varphi^g(T_1) \subset T $ and $\varphi^g(\UU_\alpha) \subset \UU$.

%can be represented by an element of , 
Let  $\lambda \in X_*(\TT_1)$ be the one-parameter $k$-subgroup of $\TT_1$ satisfying $\left<\alpha, \lambda\right> = 2$.  The $k$-split torus $\TT$ acts on the unipotent $k$-subgroup $\UU$ by conjugation. In fact $\UU$ is   $\TT$-equivariantly  $k$-isomorphic as a $k$-variety to its Lie algebra $\mathrm{Lie}(\UU)$ \cite[15.13]{borel2012linear}. 
%This implies that $ \mathrm{d}\varphi \circ \mathrm{ad}(t) = \mathrm{ad}(\varphi(t)) \circ \mathrm{d} \varphi$  implies
This implies that $$|\left<\beta,\varphi^g \circ \lambda\right>| \le  2$$
for all   $k$-roots $\beta \in \Phi_{k}(\GG,\TT)$. Here $g \in G$ is an element as in the above claim.

The compactness of   $Q_1$ implies that $Q_1 \subset K_1 Q'_1 K_1$ for some   compact subset $Q'_1 \subset T_1$.  By the  upper bound on $|\left<\beta, \varphi^g \circ \lambda\right>|$ for all $\beta \in \Phi_k(\GG,\TT)$    there is a compact subset $Q'   \subset T$ depending on $Q'_1$ such that  $\varphi^g(Q'_1) \subset Q'$. The Archimedean case  of the proof follows letting $Q = K Q' K $.

%Moreover $\mathfrak{g} = \mathfrak{t} + \mathfrak{u}_\alpha + \mathfrak{u}_{-\alpha}$.

% Since all maximal $k$-split tori of $\GG$ are conjugate by an element of $G$ \cite[\S20.9]{borel2012linear} we may conjugate $\varphi$ so that $\varphi(\TT_1)\subset \TT$. 

We now deal with the non-Archimedean case. 
 Weil's restriction of scalars allows us to assume without loss of generality that the local field $k$ is  $\Qp$ for some prime number $p$. 

  Let the  notations $\TT_1, T_1,\UU_\alpha,  \TT, T, \BB$ and $\UU$ be exactly as in the Archimedean case. Let $K_1$ be a maximal compact subgroup of $H$ admitting a Cartan decomposition $H = K_1 T_1 K_1$.  
   Let $B$ be an Iwahori subgroup \cite{iwahori1965some} of $G$ corresponding to the minimal parabolic subgroup $\BB$.   Let $K'_0, \ldots, K'_l$ be the maximal compact subgroups of $G$ containing $B$. In particular $l = \mathrm{rank}_k(\GG)$. These are precisely the representatives of  the conjugacy classes of all maximal compact subgroups in $G$. There are corresponding Cartan decompositions $G = K'_i T K'_i$ for all $i \in \{0,\ldots,l\}$. For details see e.g. \cite[3.13, 3.14]{PR}.
   
 Let $\varphi : \HH\to \GG$ be any $k$-homomorphism.  In analogy with the Archimedean case, we claim that there is an element $g \in G$ such that
$$  \varphi^{g}(K_1) \subset \bigcup_{i=0}^l K'_i, \quad \varphi^{g}(T_1) \subset T, \quad \text{and} \quad \varphi^g(\UU_\alpha) \subset \UU.$$
To see this we argue as follows. There is an element $g_1 \in G$ such that $\varphi^{g_1}(K_1) \subset K'_i$ for some $ i \in \{0,\ldots,l\}$.  Let $T'$ be any maximal $k$-split torus of $G$ containing   $\varphi^{g_1}(T_1)$. The Bruhat--Tits building associated to $G$ is strongly transitive \cite[6.56]{abramenko2008buildings}. Therefore there is some $g_2 \in B$ such that $(T')^{g_2} = T$. In particular $(K'_i)^{g_2} = K'_j$ for some $j \in \{0,\ldots,l\}$. Every element of the restricted Weyl group of the torus $\TT$ is realized in the   maximal compact subgroup $K'_j$ \cite[p. 150]{PR}.    There is an element $g_3 \in K'_j \cap \mathrm{N}_G(T)$ satisfying $\varphi^{g}(\UU_\alpha) \subset \UU$ with $g = g_3 g_2 g_1$. The claim follows.

The final part of the proof is exactly as in the Archimedean case. We may take $K' = \bigcup_{i=0}^l K'_i$ and $Q = K' Q' K'$ for some sufficiently large compact subset $Q' \subset T$ depending on $Q_1$.
\end{proof}

\section{Weakly amenable actions and Theorem \ref{thm:spectral gap action is essentially free}}
\label{sec:weakly amenable and spectral gap}
%\subsection{Weakly amenable actions do not have spectral gap}

Let $G$ be a locally compact group and $(X,\mu)$  a Borel $G$-space with an invariant probability measure. We use spectral gap to deduce  essential freeness for such an action and establish Theorem \ref{thm:spectral gap action is essentially free}.

Recall that a  $G$-space $(X,\mu)$ is \emph{weakly amenable} if the orbital equivalence relation generated by the action is amenable; see e.g.  \cite[Section 4.3]{Zi} and \cite{SZ} for details. 
%Let us quote here the a generalization of the Stuck--Zimmer theorem due to Bader--Shalom.

%due to Bader--Shalom
%
%A key ingredient in establishing essential freeness for a non weakly amenable action is the Bader--Shalom generalization of the Stuck--Zimmer theorem. %This was 
 %generalised by Bader and Shalom to products of locally compact groups.

%ssume that $G$ splits as a direct product $G = G_1 \times  G_2$. 
%Combining Spectral gap 

\begin{atheorem}[Stuck--Zimmer, Bader--Shalom]
\label{thm:BS - SZ theorem using IFT}
Let $G$ be a direct product of at least two simple groups. Assume that the $G$-space $(X,\mu)$ is  properly ergodic, irreducible and not weakly amenable. Then it is $\mu$-essentially free.
\end{atheorem}
\begin{proof}
See \cite{SZ} for the classical case of semisimple Lie groups, \cite{LIFT} for semisimple linear groups over local fields and \cite{BS} for general locally compact groups.
\end{proof}

We conclude that Theorem \ref{thm:spectral gap action is essentially free} follows at once by combining  Theorem \ref{thm:BS - SZ theorem using IFT} with the following observation, due to Creutz \cite{C_L}.

%\begin{definition*}
%A sequence of subsets $A_n \subset X$ is \emph{asymptotically invariant} if for every compact subset $Q \subset G$
%$$ \lim_n \sup_{g\in Q} \mu(gA_n \triangle A_n) = 0 $$
%An asymptotically invariant sequence of subsets is \emph{trivial} if 
%$$ \lim_n \mu(A_n) (1-\mu(A_n)) = 0$$
%The $G$-space is \emph{strongly ergodic} if every asymptotically invariant sequence is trivial.
%\end{definition*}

%\begin{prop}

%If the $G$-space $(X,\mu)$ has spectral gap then it is strongly ergodic. 
%\end{prop}
%\begin{proof}
%See \cite[6.3.2]{BdlHV}.	
%\end{proof}

\begin{aprop}
\label{prop:weakly amenable has no spectral gap}
Let $G$ be a second countable locally compact group.	
If the $G$-space $(X,\mu)$ is properly ergodic and weakly amenable then it has no spectral gap.
\end{aprop}

Roughly speaking, the proof of Proposition \ref{prop:weakly amenable has no spectral gap} is as follows. A weakly amenable action is orbit equivalent to an  action of either $\ZZ$ or $\RR$, depending on whether $G$ is countable or not. The action of these amenable groups admits more than a single invariant mean on $L^\infty(X,\mu)$. This fact is invariant under orbit equivalence, so that the action of $G$ admits more than a single  invariant mean as well. Therefore the action of $G$ is not strongly ergodic and in particular there can be  no spectral gap. 
See  \cite[7.3.1]{C_L} for complete details and  references concerning this argument.

It is interesting to note the similarity of the above theme with the role played by Selberg's theorem   in the solution of the Banach-Ruziewicz problem  \cite{lubotzky2010discrete}.

\section{Invariant random subgroups and a proof of Theorem \ref{thm:main theorem}}
\label{sec:proof of main theorem}

We now prove Theorem \ref{thm:main theorem} relying on strong uniform spectral gap for congruence lattices. 
The proof strategy is inspired by \cite{7S} and in particular invariant random subgroups\footnote{See Gelander's  lecture notes \cite{gelander2015lecture} on the Chabauty topology, invariant random subgroups and Benjamini-Schramm spaces.} play a key role. 

Associated to any second countable locally compact group $G$ is the \emph{Chabauty space}  of  closed subgroups, denoted $\Sub{G}$. This is a compact $G$-space with the Chabauty topology and the conjugation action. An \emph{invariant random subgroup} of $G$ is a $G$-invariant probability measure on $\Sub{G}$. Let $\IRS{G}$ denote the compact convex space of all invariant random subgroups of $G$ with the weak-$*$ topology. 

Associated to any normal subgroup $N \nrm G$ is the point mass $\delta_N \in \IRS{G}$.  More interestingly, to any lattice $\Gamma$ in $G$ we associate $\mu_\Gamma \in \IRS{G}$ obtained by pushing forward the $G$-invariant probability from $G/\Gamma$ to $\Sub{G}$ via the $G$-equivariant mapping
$$ G / \Gamma \to \Sub{G}, \quad g\Gamma \mapsto g\Gamma g^{-1}. $$

Let $\Gamma_n$ be a sequence of pairwise non-conjugate  lattices in $G$ and  $\mu$ any weak-$*$ accumulation point of $\mu_{\Gamma_n}$ in the space $\IRS{G}$.

\begin{prop}
\label{prop:irreducible and essentially transitive are isolated in ergodic}
If $\mu$ is irreducible and essentially transitive  then $\mu = \delta_M$ for some central subgroup $M$. 
\end{prop}

\begin{proof}
Since $(\Sub{G},\mu)$ is essentially transitive there is a closed subgroup $H \le G$  such that $\mu$ is supported on the conjugacy class of $H$ and the $G$-space $(\Sub{G},\mu)$ is isomorphic to the homogeneous $G$-space $G/N_G(H)$.  Therefore the subgroup $N_G(H)$ is of finite covolume in $G$. 

The factors $G_i$ of $G$ are non-compact and the fields $k_i$ have zero characteristic so that the  density theorem of Borel applies. As in \cite[II.6.2.(b)]{Ma} it can   used  to deduce that $N_G(H) = M \times \Gamma$ where $M\nrm G$ is a normal subgroup with  $G = M \times M'$ and   $\Gamma$ is a lattice in $M'$. Since the normal subgroup $M$ acts trivially on $G/M$   the irreducibility of the action implies  either   that     $N_G(H)$ is an irreducible lattice in $G$ or   that $N_G(H) = G$. In the second case $\mu$ is equal to $\delta_M$. 

We deal with these two possibilities separately, relying on results from \cite{GL}. 
 If $N_G(H)$ is a lattice in $G$ then $H$ must be an irreducible lattice as well by the normal subgroup theorem of Margulis \cite[IV]{Ma}. Every irreducible lattice in $G$ admits a Chabauty open neighborhood in $\Sub{G}$ consisting of conjugates \cite[1.9]{GL}, so that the corresponding point $\mu_H $ is  isolated in the space of extreme points of $\IRS{G}$.

A non-discrete normal subgroup of $G$ does not belong to the closure of the Chabauty subspace of discrete subgroups \cite[6.7]{GL}. Therefore the second case where $\mu = \delta_M$ is impossible unless $M$ is central, as required.
\end{proof}

%The congruence assumption was not used in Proposition \ref{prop:irreducible and essentially transitive are isolated in ergodic}.
%\begin{proof}[Proof of Theorem \ref{thm:main theorem}]
Assume moreover that the $\Gamma_n$'s are congruence lattices.  We claim that $\mu$ is indeed equal to $ \delta_M$ for some central subgroup $M$. Theorem \ref{thm:Clozel} implies that the sequence $\mu_{\Gamma_n}$ of $G$-invariant Borel probability measures on the compact $G$-space $\Sub{G}$ has strong uniform spectral gap. By Proposition \ref{prop:uniform spectral gap gives spectral gap in the limit} the $G$-space $(\Sub{G},\mu)$ has strong spectral gap as well. In particular $\mu$ is both irreducible and has  spectral gap. 

%Clozel's Theorem \ref{thm:Clozel} implies that $G/\Gamma(\mathfrak{q}_i)$ has strong uniform spectral gap for the factors $G_v$ with $v \in S$. This property descends to quotients,  so that the compact space $\Sub{G}$ has strong uniform spectral gap with respect to the sequence of invariant probability measures $\mu_{i_n}$ for $n \in \NN$. 

By the argument on \cite[p. 729]{SZ} we may assume to begin with that $G$ has trivial center.  The $G$-action on $(\Sub{G},\mu)$ is certainly not essentially free. The contrapositive to Theorem \ref{thm:spectral gap action is essentially free} implies that it is not properly ergodic. Therefore the $G$ action on $(\Sub{G},\mu)$  must  be essentially transitive and Proposition \ref{prop:irreducible and essentially transitive are isolated in ergodic} applies.

%\begin{prop}
%\label{prop:IRS convergence implies BS convergence}
%If $\mu_i \to \delta_{\mathrm{id}}$ then $G / \Lambda_i$ Benjamini-Schramm converges to $G$.
%\end{prop}
%\begin{proof}
%
%\end{proof}

The following three modes of convergence are all equivalent for semisimple analytic groups over local fields of zero characteristic \cite{GL} --- weak-$*$ convergence of $\mu_{\Gamma_n}$ to a central subgroup in $\IRS{G}$, Benjamini--Schramm convergence of $\Gamma_n \backslash X$ to $X$ and the fact that the sequence $\Gamma_n$ is  weakly central. 

The proof of Theorem \ref{thm:main theorem} is now complete.

\bibliography{limits}

\begin{thebibliography}{10}

\bibitem{7S}
M.~Abert, N.~Bergeron, I.~Biringer, T.~Gelander, N.~Nikolov, J.~Raimbault, and
  I.~Samet.
\newblock On the growth of {$ L^{2} $}-invariants for sequences of lattices in
  {Lie} groups.
\newblock {\em Annals of Mathematics}, 185(3):711--790, 2017.

\bibitem{abramenko2008buildings}
P.~Abramenko and K.~S. Brown.
\newblock {\em Buildings: theory and applications}, volume 248.
\newblock Springer Science \& Business Media, 2008.

\bibitem{BS}
U.~Bader and Y.~Shalom.
\newblock Factor and normal subgroup theorems for lattices in products of
  groups.
\newblock {\em Inventiones mathematicae}, 163(2):415--454, 2006.

\bibitem{BdlHV}
B.~Bekka, P.~de~La~Harpe, and A.~Valette.
\newblock {\em Kazhdan's property (T)}, volume~11.
\newblock Cambridge university press, 2008.

\bibitem{borel2012linear}
A.~Borel.
\newblock {\em Linear algebraic groups}, volume 126.
\newblock Springer Science \& Business Media, 2012.

\bibitem{burger1991ramanujan}
M.~Burger and P.~Sarnak.
\newblock Ramanujan duals {II}.
\newblock {\em Inventiones mathematicae}, 106(1):1--11, 1991.

\bibitem{clozel2003demonstration}
L.~Clozel.
\newblock D{\'e}monstration de la conjecture $\tau$.
\newblock {\em Inventiones mathematicae}, 151(2):297--328, 2003.

\bibitem{clozel2004equidistribution}
L.~Clozel and E.~Ullmo.
\newblock {\'E}quidistribution des points de {H}ecke.
\newblock {\em Contributions to automorphic forms, geometry, and number
  theory}, pages 193--254, 2004.

\bibitem{C_L}
D.~Creutz.
\newblock Stabilizers of actions of lattices in products of groups.
\newblock {\em Ergodic Theory and Dynamical Systems}, 37(4):1133--1186, 2017.

\bibitem{fraczyk2016strong}
M.~Fraczyk.
\newblock Strong limit multiplicity for arithmetic hyperbolic surfaces and $3
  $-manifolds.
\newblock {\em arXiv preprint arXiv:1612.05354}, 2016.

\bibitem{gelander2015lecture}
T.~Gelander.
\newblock A lecture on invariant random subgroups.
\newblock {\em arXiv:1503.08402}, 2015.

\bibitem{GL}
T.~Gelander and A.~Levit.
\newblock Invariant random subgroups over non-{A}rchimedean local fields.
\newblock {\em Mathematische Annalen}, 372(3-4):1503--1544, 2018.

\bibitem{gelbart1978relation}
S.~Gelbart and H.~Jacquet.
\newblock A relation between automorphic representations of {$\mathrm {GL}(2)
  $} and {$\mathrm {GL}(3) $}.
\newblock In {\em Annales scientifiques de l'{\'E}cole Normale Sup{\'e}rieure},
  volume~11, pages 471--542, 1978.

\bibitem{iwahori1965some}
N.~Iwahori and H.~Matsumoto.
\newblock On some {B}ruhat decomposition and the structure of the {H}ecke rings
  of {$p $}-adic chevalley groups.
\newblock {\em Publications Math{\'e}matiques de l'IH{\'E}S}, 25:5--48, 1965.

\bibitem{knapp2013lie}
A.~W. Knapp.
\newblock {\em Lie groups beyond an introduction}, volume 140.
\newblock Springer Science \& Business Media, 2013.

\bibitem{LIFT}
A.~Levit.
\newblock The {Nevo-Zimmer} intermediate factor theorem over local fields.
\newblock {\em Geometriae Dedicata}, 186(1):149--171, 2017.

\bibitem{lubotzky2010discrete}
A.~Lubotzky.
\newblock {\em Discrete groups, expanding graphs and invariant measures}.
\newblock Springer Science \& Business Media, 2010.

\bibitem{lubotzky2005property}
A.~Lubotzky and A.~Zuk.
\newblock On property ($\tau$).
\newblock {\em Notices Amer. Math. Soc}, 52(6):626--627, 2005.

\bibitem{Ma}
G.~A. Margulis.
\newblock {\em Discrete subgroups of semisimple {Lie} groups}, volume~17.
\newblock Springer Science \& Business Media, 1991.

\bibitem{PR}
V.~P. Platonov and A.~S. Rapinchuk.
\newblock Algebraic groups and number theory.
\newblock {\em Russian Mathematical Surveys}, 47(2):133--161, 1992.

\bibitem{raimbault2013convergence}
J.~Raimbault.
\newblock On the convergence of arithmetic orbifolds.
\newblock {\em Annales de l’institut Fourier}, 67, 11 2013.

\bibitem{selberg1965estimation}
A.~Selberg.
\newblock On the estimation of fourier coefficients of modular forms.
\newblock In {\em Proc. Sympos. Pure Math}, volume~8, pages 1--15, 1965.

\bibitem{SerCong}
J.-P. Serre.
\newblock Le probleme des groupes de congruence pour {$\mathrm{SL}_2$}.
\newblock {\em Annals of Mathematics}, 92(3):489--527, 1970.

\bibitem{SZ}
G.~Stuck and R.~J. Zimmer.
\newblock Stabilizers for ergodic actions of higher rank semisimple groups.
\newblock {\em Annals of Mathematics}, pages 723--747, 1994.

\bibitem{vigneras1983quelques}
M.-F. Vign{\'e}ras.
\newblock Quelques remarques sur la conjecture {$\lambda_1 \ge 1 / 4$},
  {S}eminar on number theory, {P}aris 1981--82.
\newblock {\em Progr. Math}, 38:321--343, 1983.

\bibitem{Zi}
R.~J. Zimmer.
\newblock {\em Ergodic theory and semisimple groups}, volume~81.
\newblock Springer Science \& Business Media, 1984.

\end{thebibliography}
\bibliographystyle{abbrv}

\end{document}